\documentclass[abstracton]{scrartcl}

\usepackage[utf8]{inputenc}
\usepackage[T1]{fontenc}
\usepackage[british]{babel}

\usepackage{geometry}

\usepackage{tikz}

\usepackage{paralist}

\usepackage[noadjust]{cite}
\usepackage{hyperref}
\usepackage{xcolor}
\hypersetup{
    colorlinks,
    linkcolor={blue!80!black},
    citecolor={blue!80!black},
    urlcolor={blue!80!black}
}
\usepackage{amsmath}
\usepackage{amssymb}
\usepackage{amsthm}
\usepackage{mathtools}
\usepackage{shuffle}
\usepackage{units}

\usepackage{lmodern}

\allowdisplaybreaks 

\newcommand{\R}{\mathbf{R}}

\newcommand{\Rd}{\mathbf{R}^d}

\newcommand{\Rn}{\mathbf{R}^n}

\newcommand{\bx}{\mathbf{x}}

\newcommand{\Zp}{\mathbf{N}_0}
\newcommand{\Zpp}{\mathbf{N}}

\newcommand{\M}{\mathcal{M}}

\newcommand{\Li}{\mathcal L}
\newcommand{\Tree}{\mathfrak {Tr}}

\newcommand{\tree}{\mathfrak t}
\newcommand{\stree}{\mathfrak s}

\newcommand{\Hi}{\mathcal{H}}

\newcommand{\id}{\mathrm{id}}
\newcommand{\eps}{\varepsilon}

\newcommand{\dd}{\mathrm d}

\newcommand{\Sym}{\mathrm{SG}}
\newcommand{\sym}{\mathrm{sym}}

\newcommand{\vertiii}[1]{{\left\vert\kern-0.25ex\left\vert\kern-0.25ex\left\vert #1 
    \right\vert\kern-0.25ex\right\vert\kern-0.25ex\right\vert}}
\newcommand{\partrans}[2]{_{#1}{/\kern-.7ex/_{\kern-.3ex #2}}}

\newcommand{\treeNodeRadius}{2pt}
\newcommand{\treeScale}{2}
\newcommand{\treeScaleSmall}{.7}

\newcommand{\tikzTree}[2]{
	\tikz[scale=\treeScale,baseline=#1,label distance=-.3em,line width=.8pt]{#2}
}

\newcommand{\tikzTreeSmall}[2]{
	\tikz[scale=\treeScaleSmall,baseline=#1,label distance=-.3em,line width=.8pt]{#2}
}

\newcommand{\addnodeLeft}[3]{
	\draw (#1){}+(-.7em,1.9ex) node[circle,fill,inner sep=\treeNodeRadius,line width=0,label=left:\footnotesize{$#3$}](#2){} -- (#1);
}
\newcommand{\addnodeMiddle}[3]{
	\draw (#1){}+(0,1.9ex) node[circle,fill,inner sep=\treeNodeRadius,line width=0,label=left:\footnotesize{$#3$}](#2){} -- (#1);
}
\newcommand{\addnodeRight}[3]{
	\draw (#1){}+(.7em,1.9ex) node[circle,fill,inner sep=\treeNodeRadius,line width=0,label=left:\footnotesize{$#3$}](#2){} -- (#1);
} 
\newcommand{\addnodeRightTwo}[3]{
	\draw (#1){}+(.35em,1.9ex) node[circle,fill,inner sep=\treeNodeRadius,line width=0,label=left:\footnotesize{$#3$}](#2){} -- (#1);
}
\newcommand{\addnodeLeftTwo}[3]{
	\draw (#1){}+(-.35em,1.9ex) node[circle,fill,inner sep=\treeNodeRadius,line width=0,label=left:\footnotesize{$#3$}](#2){} -- (#1);
}

\newcommand{\drawRoot}{
	\draw node[circle,fill,inner sep=\treeNodeRadius,line width=0](a){};
}
    
\newtheoremstyle{plaindotless}{\topsep}{\topsep}{\itshape}{}{\bfseries}{}{ }{}
\newtheoremstyle{definitiondotless}{\topsep}{\topsep}{}{}{\bfseries}{}{ }{}

\theoremstyle{definitiondotless}
\newtheorem{defn}{Definition}[section]

\newtheorem{exmp}[defn]{Example}
\newtheorem{rem}[defn]{Remark}

\theoremstyle{plaindotless}
\newtheorem{thm}[defn]{Theorem}
\newtheorem{prop}[defn]{Proposition}
\newtheorem{cor}[defn]{Corollary}
\newtheorem{lem}[defn]{Lemma}

\author{Thomas Cass\thanks{supported by EPSRC Grant EP/M00516X/1} \and Martin P. Weidner\thanks{funded by a mini-DTC scholarship from the Department of Mathematics, Imperial College London}}
\publishers{Imperial College London}
\date{}

\title{Tree algebras over topological vector spaces in rough path theory}

\begin{document}
	\maketitle
	
	\begin{abstract}
				
		We work with non-planar rooted trees which have a label set given by an arbitrary vector space $V$.  By equipping $V$ with a complete locally convex topology, we show how a natural topology is induced on the tree algebra over $V$. In this context, we introduce the Grossman-Larson and Connes-Kreimer topological Hopf algebras over $V$, and prove that they form a dual pair in a certain sense. As an application we define the class of branched rough paths over a general Banach space, and propose a new definition of a solution to a rough differential equation (RDE) driven by one of these branched rough paths. We show equivalence of our definition with a Davie-Friz-Victoir-type definition, a version of which is widely used for RDEs with geometric drivers, and we comment on applications to RDEs with manifold-valued solutions. 
		\begin{description}
			\item[Keywords:] Labelled trees, Hopf algebras, Rough differential equations
		\end{description}
	\end{abstract}
	
	\section*{Introduction}
	
A cornerstone of classical rough path theory is to understand differential
equations of the form%
\begin{equation}
\dd y_{t}=V\left(  y_{t}\right)  \dd\bx_{t}\text{ for }t\in\left[
0,T\right]  ,\text{started at }y_{0}\text{{}}\in F\text{, }\label{RDE1}%
\end{equation}
in which $\bx$ is a\emph{ geometric }$p$-\emph{rough path} on a
Banach space $E$, and $V$ is a continuous linear map from $E$ into the space
of vector fields on another Banach space $F$. A number of different approaches
to this definition exist. The first chronologically is due to Lyons
\cite{Lyons_1998} for whom a solution is defined as a fixed point, in geometric
rough path space, of the rough integral equation corresponding to
\eqref{RDE1}.~\ More recently, Gubinelli introduced an essentially equivalent
definition which characterises solutions in terms of \emph{controlled rough
paths}, cf. \cite{Gubinelli_2004,Gubinelli_2010} and \cite{Hairer_Kelly_2015}. In the finite-dimensional case $E=\Rd$ and $F=\Rn$, Davie \cite{Davie_2008} (for $2\leq p<3$) and then Friz and Victoir
\cite{Friz_Victoir_2010} (for general $p\geq2)$ proposed a notion of solution motivated by
the Euler approximation for ordinary differential equations. This definition
requires that $y$ and $\bx$ be related\footnote{The meaning of
	$\simeq$ in \eqref{davie} is made precise later, but it\ relates two terms
whose difference is negligible in a certain sense.} by  %

\begin{equation}
y_{t}-y_{s}\simeq\sum_{k=1}^{\left\lfloor p\right\rfloor }\sum_{i_{1}%
,\ldots,i_{k}=1}^{d}V_{i_{1}}\cdots V_{i_{k}}\text{id}\left(  y_{s}\right)
\bx_{s,t}^{k;i_{1},\ldots,i_{k}}\text{ for }s,t\in\left[  0,T\right]
,\label{davie}%
\end{equation}
where $\bx_{s,t}^{k;i_1,\ldots,i_k}:=\langle e_{i_1}\otimes\cdots\otimes e_{i_k},\bx_{s,t}\rangle$, and
$V_{i}\left(  y\right)  :=V\left(  y\right)  \left(  e_{i}\right)  $ when
$\left\{  e_{i}\right\}  _{i=1}^{d}$ denotes the standard basis of $\Rd$. Somewhat differently, and inspired by the work of Strichartz
\cite{Strichartz_1987}, Bailleul in \cite{Bailleul_2015} interprets a solution as the unique
\emph{rough flow} (a diffeomorphism on $F$) associated to a class
of approximate flows. This rough flow can be used to construct a path $y$
which satisfies
\begin{equation}
f\left(  y_{t}\right)  -f\left(  y_{s}\right)  \simeq\sum_{k=1}^{\left\lfloor
p\right\rfloor }\sum_{i_{1},...,i_{k}=1}^{d}V_{i_{1}}...V_{i_{k}}f\left(
y_{s}\right)  \bx_{s,t}^{k;i_{1},...,i_{k}}\text{ for }%
s,t\in\left[  0,T\right]  ,\label{bailleul}%
\end{equation}
for all smooth enough functions $f:F\rightarrow G$ and for any Banach space
$G$. See Proposition~3.5 of \cite{Bailleul_2015}.

In many situations \eqref{bailleul} is a more natural test than \eqref{davie} to
determine if a given path $y$ solves \eqref{RDE1}. For example, on any smooth
manifold \eqref{bailleul} still has meaning even
if the increment on the left side of \eqref{davie} does not. This observation has been used in \cite{Cass_Driver_Litterer_2015} in order to study rough differential equations on submanifolds of $\Rd$. One application
of this paper will be to show the equivalence of \eqref{davie} and
\eqref{bailleul}; this is currently unknown even for the special case sketched
above where $E$ and $F$ are finite-dimensional and for
geometric $\bx$. The basic challenge is algebraic: a general
smooth function ~$f$ can have non-zero derivatives of any order, while the
derivatives of the identity function of degree two and higher vanish
identically. The summands  in \eqref{bailleul} therefore typically involve many
more terms than those in \eqref{davie}, and to handle these we need to
systematically keep track of terms involving derivatives of $f$ of a fixed
degree. The~algebraic structure needed is exactly the Grossman-Larson Hopf
algebra on labelled rooted non-planar trees. These ideas, which orginally go
back to Cayley \cite{Cayley_1881}, appear in Butcher's seminal study of numerical methods
for differential equations \cite{Butcher_1972}. Connes and Kreimer \cite{Connes_Kreimer_1998} and Grossman and
Larson \cite{Grossman_Larson_1989} subsequently introduced two different Hopf algebra structures
based on rooted trees; these are dual to each another, cf. \cite{Panaite_2000} and \cite{Hoffman_2003},
allowing one to view the Butcher group either as the character group of the
Connes-Kreimer algebra, or as the group-like elements of the Grossman-Larson
algebra. Recently the connection between tree algebras, renormalisation theory and rough paths has led to new results such as the work by Bruned, Chevyrev, Friz and Prei\ss{}  \cite{Bruned_Chevyrev_Friz_Preiss_2017}.

Our analysis admits two important generalisations on the motivating discussion
above, namely:

\begin{enumerate}
\item It allows for \emph{branched rough path} drivers. These have been
introduced by Gubinelli and developed further by Hairer and Kelly. They
generalise the notion of a geometric rough path by removing the algebraic
constraints which come from the classical integration-by-parts relations. As
such, a wider class of driving signals can be accommodated, e.g. stochastic
rough paths constructed via iterated It\^{o} integration.

\item It permits $E$ and $F$ to be general Banach spaces. An key advantage of
classical rough path theory is its ability to treat models in which the signal
and response evolve in infinite dimensional state-spaces. This poses an extra
challenge to the algebraic methods, most of which work for trees with a finite
-- or at most countable -- label set. In this article, by constrast, we will
work with general locally convex topological vector spaces as label sets. 
\end{enumerate}

The eventual outcome of this paper will thus be to close two gaps in the
literature: we present the first rigorous treatment of infinite-dimensional
branched rough paths and then, in this general setting, we prove that
\eqref{bailleul} is equivalent to the classical Davie-Friz-Victoir notion of
solution to a rough differential equation.

	The structure of this work is as follows. In Section~\ref{sec:Trees} we define trees as partially ordered sets and introduce the properties and operations that are relevant in later sections. Section~\ref{sec:TreeAlgebras} contains the construction of algebras of labelled trees, where the label set is given by an arbitrary vector space $V$. If $V$ carries a topology, then this topology induces a topology on the tree algebras associated to $V$. This topology is introduced in Section~\ref{sec:Topology}. Once the algebraic and topological algebras of Grossman-Larson and Connes-Kreimer are introduced, we show in Section~\ref{sec:Duality} that they form a dual pair of Hopf algebras. Finally we apply this duality in the context of rough paths in Section~\ref{sec:RoughPaths}.

	\section{The set of non-planar rooted trees}\label{sec:Trees}
		There exist various equivalent definitions for the set of non-planar rooted trees and here we choose to follow \cite{Hoffman_2003} and work with equivalence classes of partially ordered sets. Two partially ordered sets are defined to be \emph{equivalent} if there exists an order preserving bijection between them. 
		
		\begin{defn}
			A \emph{Non-planar rooted tree} $\tree$ is an equivalence class of finite partially ordered sets $S$ such that $S$ has a unique minimal element and for every $v\in S$ the set $\{w\in S:w\prec v\}$ is totally ordered. The set of all (non-planar rooted) trees is denoted by $\Tree$.
		\end{defn}
		
		From now on, by a tree $\tree\in\Tree$  we mean a generic representative of the equivalence class and we will address the problem of well-definedness where we deem it necessary.	
	
		The elements of a tree $\tree\in\Tree$ are called \emph{vertices} and its unique minimal element is called its \emph{root} and it is denoted by $\mathrm r(\tree)$. By $\mathrm v(\tree):=\tree\backslash\mathrm r(\tree)$ we denote the set of all vertices of $\tree$ without the root. Finally we define $\lvert\tree\rvert$ to be the number of elements of $\mathrm v(\tree)$ and we call this number the \emph{degree} of $\tree$. There exists a unique $\tree\in\Tree$ with $\lvert\tree\rvert=0$ and we denote this tree by $\mathbf 1$. 	
	
	The \emph{symmetry group} $\Sym(\tree)$ of $\tree\in\Tree$ is the group of all order-preserving bijections from $\mathrm v(\tree)$ to $\mathrm v(\tree)$. 
	
	For $\tree\in\Tree$ and $v\in\tree$, denote by $\mathrm c(v)$ the set of \emph{children} of $v$, i.e. all vertices $w\in\tree$ such that $w\succ v$ and such that there exists no $x\in\tree$ with $w\succ x\succ v$. 
	
	Trees can be canonically interpreted as graphs where the vertices of the graph correspond to the vertices of the tree and the edges of the graph correspond to the statement `is a child of'.

	An important operation is that of \emph{grafting} two trees together. Let $\tree,\stree\in\Tree$ be two trees. A grafting map from $\tree$ to $\stree$ is a map $d:\mathrm c(\mathrm r(\tree))\rightarrow \stree$, i.e. it assigns a vertex of $\stree$ to every child of the root of $\tree$. Denote by $\mathrm{Gr}(\tree,\stree)$ the set of all such maps. Given $d\in\mathrm{Gr}(\tree,\stree)$ we define $\tree\stackrel d\rightharpoonup \stree:=\mathrm v(\tree)\sqcup \stree$ with the additional relation $v\succ d(v)$ for every $v\in\mathrm c(\mathrm r(\tree))$ (plus all necessary additional relations so that transitivity holds). Here $\sqcup$ denotes the disjoint union.

	 In the special case where $d(v)=\mathrm r(\stree)$ for all $v\in\mathrm c(\mathrm r(\tree))$ we write $\tree\circ\stree$ instead of $\tree\stackrel d\rightharpoonup\stree$. Observe that we have $\tree\circ\stree=\stree\circ\tree$, since the operation $\circ$ simply identifies the roots of $\tree$ and $\stree$. 
\begin{exmp}
	These are all possible ways of grafting $\tree=\tikzTree{7}{\drawRoot\addnodeLeftTwo{a}{b}{}\addnodeRightTwo{a}{c}{}}$ to $\stree=\tikzTree{7}{\drawRoot\addnodeMiddle{a}{b}{}}$:
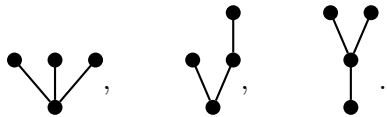
\[
	\tikzTree{7}{\drawRoot\addnodeLeft{a}{b}{}\addnodeRight{a}{c}{}\addnodeMiddle{a}{d}{}},\qquad
	\tikzTree{7}{\drawRoot\addnodeLeftTwo{a}{b}{}\addnodeRightTwo{a}{c}{}\addnodeMiddle{c}{d}{}},\qquad
	\tikzTree{7}{\drawRoot\addnodeMiddle{a}{b}{}\addnodeRightTwo{b}{c}{}\addnodeLeftTwo{b}{d}{}}.
\]
Note however that $\lvert\mathrm{Gr}(\tree,\stree)\rvert=4$ since there are two possibilities to obtain the middle one of the above three trees. Furthermore the first tree corresponds to $\tree\circ\stree$.
\end{exmp}

Let $\tree\in\Tree$ be a tree. We call a subset $C\subset \mathrm v(\tree)$ of its vertices an \emph{admissible cut} if none of its elements are comparable, i.e. if for all $v,w\in C$ we have neither $v\prec w$ nor $w\prec v$. The set of all admissible cuts of $\tree$ is denoted by $\mathrm{Cut}(\tree)$. For a given cut $C\in\mathrm{Cut}(\tree)$ we define the trees 
	 \[
	 	R^C(\tree):=\tree\backslash\{v\in \mathrm v(\tree):\text{there exists } w\in C\text{ such that }w\preceq v\}
	 \]
	 and 
	 \[
	 	P^C(\tree):=\{r\}\sqcup\{v\in \mathrm v(\tree):\text{there exists } w\in C\text{ such that }w\preceq v\}
	 \]
	 with the additional relations $r\prec w$ for all $w\in C$ (plus all necessary additional relations so that transitivity holds). Intuitively, $R^C(\tree)$ contains the tree that remains after cutting away the elements of $C$ and $P^C(\tree)$ is obtained by grafting all the subtrees that have been removed to a new root.
	 \begin{exmp}
	 	The result of the cutting operation when $C$ is given by the highlighted vertices.
	 	\[
	 		\tree=\tikzTree{15}{\drawRoot\addnodeLeft{a}{b}{}\addnodeRight{a}{c}{}\addnodeMiddle{a}{d}{}\addnodeLeftTwo{b}{e}{}\addnodeRightTwo{b}{f}{}\addnodeMiddle{c}{g}{}
	 		\draw (b) node[circle,draw,inner sep=2*\treeNodeRadius]{};
	 		\draw (g) node[circle,draw,inner sep=2*\treeNodeRadius]{};}
	 		\qquad\longrightarrow\qquad
	 		P^C(\tree)=
	 		\tikzTree{15}{\drawRoot\addnodeLeftTwo{a}{b}{}\addnodeRightTwo{a}{c}{}\addnodeLeftTwo{b}{e}{}\addnodeRightTwo{b}{f}{}},\qquad
	 		R^C(\tree)=
	 		\tikzTree{7}{\drawRoot\addnodeLeftTwo{a}{b}{}\addnodeRightTwo{a}{c}{}}
	 	\]
	 \end{exmp}
	
	In order to keep track of the vertices of a tree $\tree\in\Tree$, we would like to consider bijections $\rho:\mathrm v(\tree)\rightarrow\{1,\ldots,\lvert\tree\rvert\}$. However, since we are working with equivalence classes of trees, we have to consider equivalence classes of such bijections. Thus we assume that every representative $T$ of $\tree$ comes with a bijective map $\rho_T:\mathrm v(T)\rightarrow\{1,\ldots,\lvert\tree\rvert\}$ which we call the \emph{enumeration of $T$}. Furthermore we assume that for any two representatives $T$ and $\tilde T$ of $\tree$ the enumerations $\rho_T$ and $\rho_{\tilde T}$ are equivalent in the sense that there exists an order preserving bijection $\varphi:T\rightarrow\tilde T$ with $\rho_T=\rho_{\tilde T}\circ\varphi$.
	
	For a tree $\tree\in\Tree$ we obtain a canonical embedding of $\Sym(\tree)$ into the symmetric group $\mathfrak S_{\lvert\tree\rvert}$ by its enumeration $\rho_\tree$, i.e. we define the group
	\[
		\{\rho_\tree\pi\rho_\tree^{-1}:\pi\in\Sym(\tree)\}\subseteq\mathfrak S_{\lvert\tree\rvert}.
	\]
	We will abuse notation and denote this embedding by $\Sym(\tree)$ as well.
	 
	 For $k\in\Zpp$ we denote by $\Tree_k\subset\Tree$ the set of all trees whose root has exactly $k$ children and by $\Tree^k\subset\Tree$ the set of all trees with degree $k$. We will also use the notation $\Tree^k_n:=\Tree^k\cap\Tree_n$. The set $\Tree_1$ is especially important as the following trivial yet important lemma shows.
	 
	 \begin{lem}\label{lem:decomposition}
	Let $\tree\in\Tree_k$. Then there exist unique trees  $\tree_1,\ldots,\tree_k\in\Tree_1$ such that we have the decomposition $\tree=\tree_1\circ\cdots\circ\tree_k$.
	\end{lem}

	\section{Algebras of labelled trees over a vector space}
	\label{sec:TreeAlgebras}

		Let $M$ be a set. The most straightforward way to define a labelling of the vertices of a tree $\tree\in\Tree$ with labels from $M$ is to consider a map $l:\mathrm v(\tree)\rightarrow M$. More precisely, we have to consider equivalence classes of such maps where two maps $l$ and $\tilde l$ are equivalent if there exists $\pi\in\Sym(\tree)$ such that $l=\tilde l\pi$. We can then consider the (free) vector space spanned by all of those equivalence classes of labellings for a fixed tree $\tree\in\Tree$. Let us denote this vector space by $\Li_\tree(M)$.
		
		This definition is perfectly fine and especially if $M$ is finite there is not much difference between labelled and unlabelled trees since in that case the space $\Li_\tree(M)$ is finite dimensional for every $\tree\in\Tree$. However, if $M$ is infinite, then the vector space $\Li_\tree(M)$ is infinite dimensional for every $\tree\in\Tree\backslash\{\mathbf 1\}$. It is therefore desirable to equip $\Li_\tree(M)$ with a topology since topological vector spaces are usually nicer to work with than infinite dimensional vector spaces without a topology, for example when it comes to dual spaces. 
		
		Hence we would like to establish a different point of view which takes the (free) vector space $V$ spanned by $M$ as the label set rather than the set $M$ itself. In fact, it is usually $V$ that is given in the first place rather than $M$. Thus our aim is to describe the spaces $\Li_\tree(M)$ in terms of $V$ rather than $M$ so that we can then use a given topology on $V$ to construct a topology on $\Li_\tree(M)$. 
		
		In terms of the application to rough paths that we will present in Section~\ref{sec:RoughPaths} the vector space $V$ will take the role of the noise space. The topological tree algebra generated by it -- in the way given below -- will then be the space in which the rough path lift of the noise takes its values.		
		
		\subsection{Spaces of labelled trees as tensor products}
		Let now $V$ be a given vector space. We will perform a purely algebraic construction in this section and then consider a topological vector space in the following one. A permutation $\sigma\in\mathfrak S_k$ with $k\in\Zpp$ can be canonically interpreted as a linear map on $V^{\otimes k}$ which is given by $\sigma(v_1\otimes\cdots\otimes v_k)=v_{\sigma(1)}\otimes\cdots\otimes v_{\sigma(k)}$ for elementary tensors. Given a subgroup $G\subset\mathfrak S_k$  we define the projection
	\[
		\sym_G:V^{\otimes k}\rightarrow V^{\otimes k}:a\mapsto \frac1{\lvert G\rvert}\sum_{\sigma\in G} \sigma(a).
	\]

	\begin{defn}
		Let $V$ be a vector space, let $\tree\in\Tree$ be a tree and let $\Sym(\tree)$ be its symmetry group. Then the \emph{$\tree$-tensor power of $V$} is defined as the quotient space
		\[
			V^{\otimes\tree}:=V^{\otimes\lvert\tree\rvert}\left/\mathrm{ker}\left(\sym_{\Sym(\tree)}\right)\right.
		\]
	\end{defn}
	
	The space $V^{\otimes\tree}$ is isomorphic to the image of $\sym_{\Sym(\tree)}$, which provides a more direct  description of $V^{\otimes\tree}$. We say that an element $\mathbf t\in V^{\otimes\tree}$ is an  \emph{elementary labelled tree} if its equivalence class contains an elementary tensor. The set of all elementary labelled trees in $V^{\otimes\tree}$ spans $V^{\otimes\tree}$. 
	
	To illustrate an elementary labelled tree we choose one of its elementary representatives as in the following example. E.g., for\renewcommand{\treeNodeRadius}{1pt} $\tree=\kern-.5em\tikzTreeSmall{0}{\drawRoot\addnodeLeftTwo{a}{b}{}\addnodeRightTwo{a}{c}{}}$\renewcommand{\treeNodeRadius}{2pt} we can visualise $\mathbf t=[2v_1\otimes v_2+v_2\otimes v_1]\in V^{\otimes\tree}$ as 
	\[
		\mathbf t=3\kern-1em\tikzTree{5}{\drawRoot\addnodeLeftTwo{a}{b}{v_1}\addnodeRightTwo{a}{c}{v_2}}.
	\]
	
	Let us now define the following spaces.
	\begin{align*}
		\Tree^k(V)&:=\bigoplus_{\tree\in\Tree^k}V^{\otimes\tree}\text{ for }k\in\Zpp\\
		\Tree_k(V)&:=\bigoplus_{\tree\in\Tree_k}V^{\otimes\tree}\text{ for }k\in\Zpp\\
		\Tree^k_n(V)&:=\bigoplus_{\tree\in\Tree^k_n}V^{\otimes\tree}=\Tree^k(V)\cap\Tree_n(V)\text{ for }k,n\in\Zpp\\
		\Tree(V)&:=\bigoplus_{\tree\in\Tree}V^{\otimes\tree}=\bigoplus_{k=0}^\infty\Tree^k(V)=\bigoplus_{k=0}^\infty\Tree_k(V)=\bigoplus_{k=0}^\infty\bigoplus_{n=0}^\infty\Tree^k_n(V)
	\end{align*}
	
	The definition of $\Tree(V)$ provides canonical projections onto $\Tree^k(V)$, $\Tree_k(V)$ and $\Tree^k_n(V)$ for $k,n\in\Zpp$ and we will denote those projections by $\pi^k$, $\pi_k$ and $\pi^k_n$ respectively.
	
	We will also need truncations of the space $\Tree(V)$. Thus for $k\in\Zpp$ define the subspace 
	\[
		I^{(k)}:=\bigoplus_{\tree\in\Tree:\lvert\tree\rvert>k}V^{\hat\otimes\tree}
	\]
	and the quotient space
	\[
		\Tree^{(k)}(V):=\Tree(V)\left/I^{(k)}\right..
	\]
	The canonical projection $\Tree(V)\rightarrow\Tree^{(k)}(V)$ will be denoted by $\pi^{(k)}$. We will furthermore need the projection $\pi^{(k)}_n:=\pi^{(k)}\circ \pi_n=\pi_n\circ \pi^{(k)}$.
	
	Before we can define the additional algebraic structure on $\Tree(V)$ to turn it into a Hopf algebra we have to extend two auxiliary constructions from the previous section, namely grafting and cutting, to tensors. We start with grafting.
	\begin{defn}
		\label{def:labelledGraft}
		Let $\tree,\stree\in\Tree$ and let $t\in V^{\otimes\lvert\tree\rvert}$ and $s\in V^{\otimes\lvert\stree\rvert}$ be two elementary tensors. Let furthermore $d\in\mathrm{Gr}(\tree,\stree)$ be a grafting map. Then we define
		\begin{equation}\label{eq:grafting}
			t\stackrel d\rightharpoonup s:=\left(\rho_{\tree,\stree}\circ\rho_{\tree\stackrel d\rightharpoonup \stree}^{-1}\right)(t\otimes s)\in V^{\otimes(\lvert\tree\rvert+\lvert\stree\rvert)},
		\end{equation}
		where
		\begin{equation}\label{eq:enumConcat}
			\rho_{\tree,\stree}:\mathrm v(\tree)\sqcup\mathrm v(\stree)\rightarrow\{1,\ldots,\lvert\tree\rvert+\lvert\stree\rvert\}:v\mapsto\begin{cases}\rho_\tree(v)&v\in\mathrm v(\tree),\\\rho_\stree(v)+\lvert\tree\rvert&v\in\mathrm v(\stree).\end{cases}
		\end{equation}
		
	\end{defn}
	\begin{rem}\label{rem:dependenceOnTS}
		This definition has to be handled with care since the value of $t\stackrel d\rightharpoonup s$ may depend on the choice of representatives of $\tree$ and $\stree$. However, this is unproblematic because it is only an auxiliary construction for Definition~\ref{defn:coproducts}, where the resulting object will be independent of the choice of representatives, which will be shown in Lemma~\ref{lem:ChoiceOfReps}.
	\end{rem}
	
	By analogy with to the previous section, if $d\in\mathrm{Gr}(\tree,\stree)$ maps all elements of $\mathrm c(\mathrm r(\tree))$ to $\mathrm r(\stree)$ we write $t\circ s$ instead of $t\stackrel d\rightharpoonup s$.
	\begin{lem}
		Let $\tree,\stree\in\Tree$ and let $\mathbf t\in V^{\otimes\tree}$ and $\mathbf s\in V^{\otimes\stree}$ be two labelled trees. Let furthermore $t\in V^{\otimes\lvert\tree\rvert}$ and $s\in V^{\otimes\lvert\stree\rvert}$ be representatives of $\mathbf t$ and $\mathbf s$ respectively. Then the equivalence class $[t\circ s]\in V^{\otimes(\tree\,\circ\,\stree)}$ depends neither on the choice of $t$ and $s$ nor on the choice of representatives of $\tree$ and $\stree$.
	\end{lem}	
	\begin{proof}
		Choosing a different representative of $\tree$  and/or a different choice of $t$ amount to replacing $t$ by $\tilde t$ with the property that
		\[
			\sym_{\Sym(\tree)}(t-\tilde t) = 0.
		\]
		From \eqref{eq:grafting} we get $t\circ s-\tilde t\circ s=(t-\tilde t)\circ s$.
		Thus if $\mathrm{SG}(\tree)$ embeds as a subgroup into $\mathrm{SG}(\tree\circ\stree)$ we are done because then we have
		\[
			[t\circ s]-[\tilde t\circ s] = [(t-\tilde t)\circ s]=0.
		\]
		But this is clearly the case because the operation $\circ$ identifies the roots of $\tree$ and $\stree$ and hence both $\tree$ is a subtree of $\tree\circ\stree$.
		
		The same argument shows that choosing a different representative of $\stree$ and/or a different choice of $s$ does not change $[t\circ s]$.
	\end{proof}
	The lemma allows us to define $\mathbf t\circ\mathbf s:=[t\circ s]$ for representatives $t$ and $s$ of $\mathbf t$ and $\mathbf s$ respectively.
	If $\mathbf t\in V^{\otimes\tree}$ is an elementary labelled tree as defined above, then the unique decomposition $\tree=\tree_1\circ\cdots\circ\tree_k$ from Lemma~\ref{lem:decomposition} induces a decomposition 
	\begin{equation}\label{eq:decompositionLabelled}
		\mathbf t=\mathbf t_1\circ\cdots\circ\mathbf t_k\text{ with }\mathbf t_1\in V^{\otimes\tree_1},\ldots,\mathbf t_k\in V^{\otimes\tree_k}.
	\end{equation}

	Let us now turn to cutting. Let $\tree\in\Tree$ and $C\in\mathrm{Cut}(\tree)$. Then the vertices of $R^C(\tree)$ and $P^C(\tree)$ are defined as subsets of the vertices of $\tree$ and hence there are canonical injections 
	\[
		\phi_{C,R}:\mathrm v\left(R^C(\tree)\right)\rightarrow\mathrm v(\tree)\qquad\text{and}\qquad\phi_{C,P}:\mathrm v\left(P^C(\tree)\right)\rightarrow\mathrm v(\tree).
	\] 
	\begin{defn}
		Let $\tree\in\Tree$ be a tree and let $C\in\mathrm{Cut(\tree)}$ be an admissible cut. Then we define
		\[
			\mathbf R^C:V^{\otimes \lvert\tree\rvert}\rightarrow V^{\otimes \lvert R^C(\tree)\rvert}:v_1\otimes\cdots\otimes v_{\lvert\tree\rvert}\mapsto v_{\varphi_{C,R}(1)}\otimes\cdots\otimes v_{\varphi_{C,R}(\lvert R^C(\tree)\rvert)}
		\]
		and
		\[
		\mathbf P^C:V^{\otimes \lvert\tree\rvert}\rightarrow V^{\otimes \lvert P^C(\tree)\rvert}:v_1\otimes\cdots\otimes v_{\lvert\tree\rvert}\mapsto v_{\varphi_{C,P}(1)}\otimes\cdots\otimes v_{\varphi_{C,P}(\lvert P^C(\tree)\rvert)},
		\]
		where $\varphi_{C,R}:=\rho_\tree\circ\phi_{C,R}\circ\rho_{R^C(\tree)}^{-1}$ and $\varphi_{C,P}:=\rho_\tree\circ\phi_{C,P}\circ\rho_{P^C(\tree)}^{-1}$.
	\end{defn}
	\begin{rem}
		As in Remark~\ref{rem:dependenceOnTS} we should emphasize that $\mathbf R^C$ and $\mathbf P^C$ depend on the choice of the representative of $\tree$. Nevertheless, the following definition makes sense as Lemma~\ref{lem:ChoiceOfReps} shows.
	\end{rem}
	\begin{defn}\label{defn:coproducts}
		Let $V$ be a vector space, let $\tree,\stree\in\Tree$ and let $\mathbf t\in V^{\otimes\tree}$ and $\mathbf s\in V^{\otimes\stree}$ be elementary labelled trees. 
		\begin{enumerate}[(i)]
		\item For elementary representatives $t$ and $s$ of $\mathbf t$ and $\mathbf s$ we define	
		\[
			\mathbf t\star\mathbf s:=\sum_{\mathclap{d\in\mathrm{Gr}(\tree,\stree)}}\left[ t\stackrel d\rightharpoonup s\right]\in\Tree(V).
		\]
		\item
		If $\mathbf t$ can be decomposed into $\mathbf t=\mathbf t_1\circ\cdots\circ\mathbf t_k$ as in \eqref{eq:decompositionLabelled} we define
		\[
			\Delta_{\mathrm{GL}}(\mathbf t):=\sum_{I\subset\{1,\ldots,k\}}\mathop\bigcirc_{i\in I}\mathbf t_i\otimes\mathop\bigcirc_{\mathclap{i\in I^{\mathrm c}}}\mathbf t_i\in\Tree(V)\otimes\Tree(V),
		\]
		where the sum ranges over all subsets of $\{1,\ldots,k\}$ and we set $\mathop\bigcirc_{i\in \emptyset}\mathbf t_i=\mathbf 1$. 
		 \item For an elementary representative $t$ of $\mathbf t$ we define
		\[
			\Delta_{\mathrm{CK}}(\mathbf t):=\sum_{C\in\mathrm{Cut(\tree)}}\left[\mathbf P^C(t)\right]\otimes\left[\mathbf R^C(t)\right]\in\Tree(V)\otimes\Tree(V).
		\]
		\item
		For $\lambda\in\R$ we define
		\[
			\eta(\lambda):=\lambda\mathbf 1\in\Tree(V)
		\]
		\item
		Finally we define
		\[
			\eps(\mathbf t):=\begin{cases}\mathbf t,&\tree=\mathbf 1,\\0,&\text{otherwise}\end{cases}\in\R.
		\]
		\end{enumerate}
		All of the above maps are then extended linearly to all $\mathbf t,\mathbf s\in\Tree(V)$ and $\lambda\in\R$.
	\end{defn}
	
	\begin{exmp}
		\begin{align*}
			\tikzTree{5}{\drawRoot\addnodeLeftTwo{a}{b}{v_1}\addnodeRightTwo{a}{c}{v_2}}\ \ \star\tikzTree{5}{\drawRoot\addnodeMiddle{a}{b}{v_3}}
			\ \ =&\kern-1em
			\tikzTree{5}{\drawRoot\addnodeLeft{a}{b}{v_1}\addnodeRight{a}{c}{v_3}\addnodeMiddle{a}{d}{v_2}}\ \ +
	\tikzTree{15}{\drawRoot\addnodeLeftTwo{a}{b}{v_1}\addnodeRightTwo{a}{c}{v_3}\addnodeMiddle{c}{d}{v_2}}\ \ +
	\tikzTree{15}{\drawRoot\addnodeLeftTwo{a}{b}{v_2}\addnodeRightTwo{a}{c}{v_3}\addnodeMiddle{c}{d}{v_1}}\ \ +\kern-.8em
	\tikzTree{15}{\drawRoot\addnodeMiddle{a}{b}{v_3}\addnodeRightTwo{b}{c}{v_1}\addnodeLeftTwo{b}{d}{v_2}}\\
			\tikzTree{5}{\drawRoot\addnodeLeftTwo{a}{b}{v_1}\addnodeRightTwo{a}{c}{v_2}}\ \ \circ\tikzTree{5}{\drawRoot\addnodeMiddle{a}{b}{v_3}}
			\ \ =&\kern-1em
			\tikzTree{5}{\drawRoot\addnodeLeft{a}{b}{v_1}\addnodeRight{a}{c}{v_3}\addnodeMiddle{a}{d}{v_2}}\\
			\Delta_{\mathrm{GL}}\left(\kern-.6em\tikzTree{15}{\drawRoot\addnodeLeftTwo{a}{b}{v_1}\addnodeRightTwo{a}{c}{v_3}\addnodeMiddle{c}{d}{v_2}}\right)
			\ \ =&\ 
			\tikzTree{15}{\drawRoot\addnodeLeftTwo{a}{b}{v_1}\addnodeRightTwo{a}{c}{v_3}\addnodeMiddle{c}{d}{v_2}}\otimes\tikzTree{-3}{\drawRoot}
			\ \ +
			\tikzTree{5}{\drawRoot\addnodeMiddle{a}{b}{v_1}}
			\otimes\kern-.3em
			\tikzTree{15}{\drawRoot\addnodeMiddle{a}{c}{v_3}\addnodeMiddle{c}{d}{v_2}}		
			\ \ +\ 
			\tikzTree{15}{\drawRoot\addnodeMiddle{a}{c}{v_3}\addnodeMiddle{c}{d}{v_2}}		
			\,\otimes\kern-.7em
			\tikzTree{5}{\drawRoot\addnodeMiddle{a}{b}{v_1}}
			\ \ +\ \ 
			\tikzTree{-3}{\drawRoot}\otimes\kern-.3em
			\tikzTree{15}{\drawRoot\addnodeLeftTwo{a}{b}{v_1}\addnodeRightTwo{a}{c}{v_3}\addnodeMiddle{c}{d}{v_2}}\\
			\Delta_{\mathrm{CK}}\left(\kern-.6em\tikzTree{15}{\drawRoot\addnodeLeftTwo{a}{b}{v_1}\addnodeRightTwo{a}{c}{v_3}\addnodeMiddle{c}{d}{v_2}}\right)
			\ \ =&\ 
			\tikzTree{15}{\drawRoot\addnodeLeftTwo{a}{b}{v_1}\addnodeRightTwo{a}{c}{v_3}\addnodeMiddle{c}{d}{v_2}}\otimes\tikzTree{-3}{\drawRoot}
			\ \ +
			\tikzTree{5}{\drawRoot\addnodeMiddle{a}{b}{v_1}}
			\otimes\kern-.3em
			\tikzTree{15}{\drawRoot\addnodeMiddle{a}{c}{v_3}\addnodeMiddle{c}{d}{v_2}}		
			\ \ +\ 
			\tikzTree{15}{\drawRoot\addnodeMiddle{a}{c}{v_3}\addnodeMiddle{c}{d}{v_2}}		
			\,\otimes\kern-.7em
			\tikzTree{5}{\drawRoot\addnodeMiddle{a}{b}{v_1}}
			\ \ +\ \ 
			\tikzTree{-3}{\drawRoot}\otimes\kern-.3em
			\tikzTree{15}{\drawRoot\addnodeLeftTwo{a}{b}{v_1}\addnodeRightTwo{a}{c}{v_3}\addnodeMiddle{c}{d}{v_2}}
			\\&+		
			\tikzTree{5}{\drawRoot\addnodeLeftTwo{a}{b}{v_1}\addnodeRightTwo{a}{c}{v_2}}\otimes\kern-.7em\tikzTree{5}{\drawRoot\addnodeMiddle{a}{b}{v_3}}	
			\ \ +
			\tikzTree{5}{\drawRoot\addnodeMiddle{a}{c}{v_2}}\otimes\kern-.7em\tikzTree{5}{\drawRoot\addnodeRightTwo{a}{b}{v_3}\addnodeLeftTwo{a}{c}{v_1}}	
		\end{align*}
	\end{exmp}	
	
	\begin{lem}\label{lem:ChoiceOfReps}
		With the notation as in the previous definition, neither $\mathbf t\star\mathbf s$ nor $\Delta_{\mathrm{CK}}(\mathbf t)$ depends on the choice of $t$ and $s$ or on the choice of representatives of $\tree$ and $\stree$.
	\end{lem}
	\begin{proof}

		Let $\tilde t$ and $\tilde s$ be different elementary representatives of $\mathbf t$ and $\mathbf s$ respectively. The simple yet crucial observation is that there exist 
		$\sigma_1\in\Sym\left(\tree\right)$ and $\sigma_2\in\Sym\left(\stree\right)$
		 such that we have

		\[
		\left[\tilde t\stackrel{d}\rightharpoonup \tilde s\right]
		=\left[\sigma_1(t)\stackrel{d}\rightharpoonup\sigma_2(s)\right]
		=\left[t\stackrel {\sigma_2\,\circ\,d\,\circ\,\sigma_1}\rightharpoonup s\right],
		\]
		where the second equality is a straightforward consequence of Definition~\ref{def:labelledGraft} and the fact that the enumerations of different representatives of the same tree are compatible.
		Here $\sigma_1$ and $\sigma_2$ may also depend on the choice of representatives of $\tree$ and $\stree$. On the other hand, for fixed $\sigma_1\in\Sym(\tree)$ and $\sigma_2\in\Sym(\stree)$ we have
		\[
			\mathrm{Gr}(\tree,\stree)=\{\sigma_2\circ d\circ\sigma_1:d\in\mathrm{Gr}(\tree,\stree)\}.
		\]
		Hence $\mathbf t\star\mathbf s$ is well-defined as long as the same representatives are chosen for each term in the defining sum.
		  
		  A similar argument shows that $\Delta_{\mathrm{CK}}(\mathbf t)$ is well-defined since we sum over all admissible cuts.
	\end{proof}
	
	\begin{prop}\label{prop:algebraicHopf}
		Each of the tuples
	\begin{align*}
		\Hi_{\mathrm{GL}}(V)&:=(\Tree(V),\star,\eta,\Delta_{GL},\eps)\text{ and}\\
		\Hi_{\mathrm{CK}}(V)&:=(\Tree(V),\circ,\eta,\Delta_{CK},\eps)
	\end{align*}	
	defines a graded bialgebra and therefore each of them can be equipped with an antipode which turns them into Hopf algebras. We call $\Hi_{GL}(V)$ the \emph{Grossman-Larson Hopf algebra over $V$} and $\Hi_{CK}(V)$ the \emph{Connes-Kreimer Hopf algebra over $V$}.
	\end{prop}
	\begin{proof} 
		Let us show that the product $\star$ is associative. We will essentially follow \cite{Grossman_Larson_1989}. Let thus $\tree_1,\tree_2,\tree_3\in\Tree$ and $\mathbf t_1\in V^{\otimes\tree_1}$, $\mathbf t_2\in V^{\otimes\tree_2}$ and $\mathbf t_3\in V^{\otimes\tree_3}$ and fix representatives $t_1$, $t_2$ and $t_2$ of $\mathbf t_1$, $\mathbf t_2$ and $\mathbf t_3$ respectively. Then we need to show
		\begin{equation}\label{eq:GraftingAssociative}
			\sum_{d_1\in\mathrm{Gr}(\tree_1,\tree_2)}\sum_{d_2\in\mathrm{Gr}(\tree_1\stackrel{d_1}\rightharpoonup\tree_2,\tree_3)} \left[\left(t_1\stackrel{d_1}\rightharpoonup t_2\right)\stackrel{d_2}\rightharpoonup t_3\right]
			=
			\sum_{e_2\in\mathrm{Gr}(\tree_2,\tree_3)}\sum_{e_1\in\mathrm{Gr}(\tree_1,\tree_2\stackrel{e_2}\rightharpoonup\tree_3)}\left[ t_1\stackrel{e_1}\rightharpoonup \left(t_2\stackrel{e_2}\rightharpoonup t_3\right)\right].
		\end{equation}
		Thus assume $d_1\in\mathrm{Gr}(\tree_1,\tree_2)$ and $d_2\in\mathrm{Gr}(\tree_1\stackrel{d_1}\rightharpoonup\tree_2,\tree_3)$. Define $e_2\in\mathrm{Gr}(\tree_2,\tree_3)$ to be the restriction of $d_2$ to the children of the root of $\tree_2$ and define $e_1\in\mathrm{Gr}(\tree_1,\tree_2\stackrel{e_2}\rightharpoonup\tree_3)$ via
		\[
			e_1(v):=\begin{cases}d_1(v)&d_1(v)\neq\mathrm{r}(\tree_2)\\d_2(v)&d_1(v)=\mathrm{r}(\tree_2)\end{cases}.
		\]
		It follows from \cite[Lem. 2.6/2.7]{Grossman_Larson_1989} that this gives a one-to-one correspondence between the terms on both sides of \eqref{eq:GraftingAssociative} such that
		\[
		\left(\tree_1\stackrel{d_1}\rightharpoonup \tree_2\right)\stackrel{d_2}\rightharpoonup \tree_3= \tree_1\stackrel{e_1}\rightharpoonup \left(\tree_2\stackrel{e_2}\rightharpoonup \tree_3\right)
		\]
		holds. In light of \eqref{eq:enumConcat} we define
		\begin{align*}
			\rho_{\tree_1,\tree_2,\tree_3}:\mathrm{v}(\tree_1)\sqcup\mathrm{v}(\tree_2)\sqcup\mathrm{v}(\tree_3)&\rightarrow\{1,\ldots,\lvert\tree_1\rvert+\lvert\tree_2\rvert+\lvert\tree_3\rvert\}\\
			v&\mapsto
			\begin{cases}
			\rho_{\tree_1}(v)&v\in\mathrm{v}(\tree_1)\\
			\rho_{\tree_2}(v)+\lvert\tree_1\rvert&v\in\mathrm{v}(\tree_2)\\
			\rho_{\tree_3}(v)+\lvert\tree_1\rvert+\lvert\tree_2\rvert&v\in\mathrm{v}(\tree_3)\\
			\end{cases}
		\end{align*}
		and it is then easy to verify that 
		\[
			\left(t_1\stackrel{d_1}\rightharpoonup t_2\right)\stackrel{d_2}\rightharpoonup t_3= t_1\stackrel{e_1}\rightharpoonup \left(t_2\stackrel{e_2}\rightharpoonup t_3\right)=\left(\rho_{\tree_1,\tree_2,\tree_2}\circ\rho_{\left(\tree_1\stackrel{d_1}\rightharpoonup \tree_2\right)\stackrel{d_2}\rightharpoonup \tree_3}^{-1}\right)(t_1\otimes t_2\otimes t_3)
		\]
		where $d_1,d_2$ and $e_1,e_2$ are related as above. Hence associativity of $\star$ follows.
		
	In a similar way one can verify that the remaining (co)associativity and compatibility conditions hold.
	\end{proof}
	
	Observe that $I^{(k)}$ as defined above is an ideal for both $\star$ and $\circ$. Hence $\Tree^{(k)}$ can be interpreted as a quotient algebra of both $\left(\Tree(V),\star\right)$ and $\left(\Tree(V),\circ\right)$. On the other hand $I^{(k)}$ is a coideal for neither $\Delta_{\mathrm{GK}}$ nor $\Delta_{\mathrm{CL}}$ and hence the quotient is not a quotient Hopf algebra. Furthermore $\Tree^{(k)}(V)$ is isomorphic to the closed subspace 
	\[
		\bigoplus_{\tree\in\Tree:\lvert\tree\rvert\leq k}V^{\otimes\tree}
	\]
	of $\Tree(V)$, which justifies the name \emph{truncated tree algebra}. We will use this isomorphism implicitly without mentioning it. Note however that this vector space isomorphism is not an algebra morphism.
	
	\section{Algebras of labelled trees over a topological vector space}
	\label{sec:Topology}
	Let $V$ be a complete locally convex vector space. Our aim is to understand how the topology of $V$ can induce a topology on $\Tree(V)$. This question can be split into two problems. First we have to introduce a topology on each of the spaces $V^{\otimes\tree}$ for $\tree\in\Tree$ and second we have to choose an appropriate topology on their direct sum.

	Let $\tree\in\Tree$ be a tree. The space $V^{\otimes\tree}$ is (isomorphic to) a quotient of a tensor power of $V$ and therefore we can use the well-developed theory of topological tensor products. We recommend \cite{Grothendieck_1955} for the necessary background. Most importantly, there is no canonical topology on $V^{\otimes k}$ for $k\in\Zpp$. However, if $(\xi_i)_{i\in I}$ is a family of seminorms that generates the topology of $V$, then for every $k\in\Zpp$ and indices $i_1,\ldots,i_k\in I$ there exists a seminorm $\xi_{i_1}\otimes\cdots\otimes\xi_{i_k}$ on $V^{\otimes k}$ such that the system of seminorms thus obtained satisfies
	\[
		\xi_{i_1}\otimes\cdots\otimes\xi_{i_{m+n}}(a\otimes b)=\xi_{i_1}\otimes\cdots\otimes\xi_{i_m}(a)\xi_{i_{m+1}}\otimes\cdots\otimes\xi_{i_{m+n}}(b)
	\]
	for all $m,n\in\Zpp$, $a\in V^{\otimes m}$, $b\in V^{\otimes n}$ and $i_1,\ldots,i_{m+n}\in I$. A system of seminorms with this property is called a system of \emph{cross (semi)norms}. This system is in general not unique and the most important choices are the \emph{projective seminorms}, given by
	\[
		\xi_{i_1}\otimes\cdots\otimes\xi_{i_k}(a):=\inf\left\{\sum_{j=1}^N\xi_{i_1}(v^j_1)\cdots\xi_{i_k}(v^j_k):a=\sum_{j=1}^N v_1^j\otimes\cdots\otimes v_k^j\right\}
	\]
	and the \emph{injective seminorms}, given by
	\[
		\xi_{i_1}\otimes\cdots\otimes\xi_{i_k}(a):=\sup\left\{\lvert(v_1'\otimes\cdots\otimes v'_k)(a)\rvert:v_1',\ldots,v_k'\in \kern-.1em V'\kern-.2em,\xi_{i_1}(v_1')=\cdots=\xi_{i_k}(v_k')=1\right\}
	\]
	for all $a\in V^{\otimes k}$. Both the projective and the injective seminorms are symmetric in the sense of the following definition.
	
	\begin{defn}
		Let $V$ be a locally convex vector space and let $(\xi_i)_{i\in I}$ be a family of seminorms that generates the topology of $V$. A system of cross seminorms on the tensor powers of $V$ is called \emph{symmetric} if for all $k\in\Zpp$, all $i_1,\ldots,i_k\in I$, all $a\in V^{\otimes k}$ and all $\sigma\in\mathfrak S_k$ we have
		\[
			\xi_1\otimes\cdots\otimes\xi_k(a)=\xi_{\sigma(1)}\otimes\cdots\otimes\xi_{\sigma(k)}(\sigma(a)).
		\]
	\end{defn}
		 
	From now on we will always assume that the tensor powers of $V$ are equipped with a symmetric system of cross seminorms and for all $k\in\Zpp$ we will denote by $V^{\hat\otimes k}$ the completion of $V^{\otimes k}$ with respect to the respective family of seminorms. These assumptions imply that for all trees $\tree\in\Tree$ the projection $\sym_{\Sym(\tree)}$ defined above is continuous. Therefore it can be uniquely extended to a continuous map on $V^{\hat\otimes\lvert\tree\rvert}$ and we define
	\[	
		V^{\hat\otimes\tree}:=V^{\hat\otimes\lvert\tree\rvert}\left/\mathrm{ker}\left(\sym_{\Sym(\tree))}\right)\right..
	\]
	 Equivalently, $V^{\hat\otimes\tree}$ is the completion of $V^{\otimes\tree}$ under the quotient topology.

	\[
		\widehat\Tree(V):=\mathrm{cl}\left(\bigoplus_{\tree\in\Tree}V^{\hat\otimes\tree}\right),
	\]
	where the closure is taken with respect to the product topology. 	Furthermore we observe that for $\tree,\stree\in\Tree$ we have an isomorphism
	\[
		V^{\otimes\tree}\otimes V^{\otimes\stree}\simeq\left(V^{\otimes\lvert\tree\rvert}\otimes V^{\otimes\lvert\stree\rvert}\right)\left/\left(V^{\otimes\lvert\tree\rvert}\otimes\mathrm{ker}\left(\sym_{\Sym(\stree)}\right)+\mathrm{ker}\left(\sym_{\Sym(\tree)}\right)\otimes V^{\otimes\lvert\stree\rvert}\right)\right.
	\]
	and therefore the system of cross-seminorms that we have chosen for the $V^{\otimes k}$ gives a locally convex topology on $V^{\otimes\tree}\otimes V^{\otimes\stree}$. Thus, using distributivity of the tensor product and direct sums, we can define the completion of $\widehat\Tree(V)\otimes\widehat\Tree(V)$ via
	\[
		\widehat\Tree(V)\hat\otimes\widehat\Tree(V)=\mathrm{cl}\left(\bigoplus_{\tree,\stree\in\Tree}\left(V^{\hat\otimes\tree}\hat\otimes V^{\hat\otimes\stree}\right)\right).
	\]	
	We can now state a topological version of Proposition~\ref{prop:algebraicHopf}.
	
	\begin{prop}
		Each of the tuples
	\begin{align*}
		\widehat\Hi_{\mathrm{GL}}(V)&:=\left(\widehat\Tree(V),\star,\eta,\Delta_{GL},\eps\right)\text{ and}\\
		\widehat\Hi_{\mathrm{CK}}(V)&:=\left(\widehat\Tree(V),\circ,\eta,\Delta_{CK},\eps\right)
	\end{align*}	
	defines a topological graded bialgebra and therefore each of them can be equipped with an antipode which turns them into topological Hopf algebras. We call $\Hi_{GL}(V)$ the \emph{topological Grossman-Larson Hopf algebra over $V$} and $\Hi_{CK}(V)$ the \emph{topological Connes-Kreimer Hopf algebra over $V$}.
	\end{prop}
	\begin{proof}
		It is immediately clear that all operations of $\widehat\Hi_{\mathrm{GL}}(V)$ and $\widehat\Hi_{\mathrm{CK}}(V)$ are continuous since they only consist of tensor multiplications and permutations. Hence the claim follows directly from Proposition~\ref{prop:algebraicHopf}.
	\end{proof}
	\begin{rem}
		Technically, a topological Hopf algebra $\Hi$ is not a Hopf algebra since its coproduct may take values in $\Hi\hat\otimes\Hi$ rather than $\Hi\otimes\Hi$. Nevertheless, when we write \emph{Hopf algebra} in the sequel, we always mean \emph{Hopf algebra or topological Hopf algebra}. 
	\end{rem}

	\begin{rem}
		There are many sensible choices for a topology on $\bigoplus_{\tree\in\Tree}V^{\hat\otimes\tree}$. For instance, one could consider a topology akin to the one that is introduced in \cite{Chevyrev_Lyons_2016} for the tensor algebra. If $V$ is Banach, considering the $l^p$ direct sum for $p\in[1,\infty]$ might be a useful choice. We have chosen the product topology because it is the coarsest sensible choice (in the sense that it renders the canonical projections continuous) and hence it leads to the largest possible closure (which turns out to be the direct product $\prod_{\tree\in\Tree}V^{\hat\otimes\tree}$). Therefore the separation results in the next section hold automatically for any other reasonable choice of topology.
	\end{rem}

	\section{The Grossman-Larson and Connes-Kreimer algebras form a dual pair}
	\label{sec:Duality}
	
	In the unlabelled case (which corresponds to $V=\R$ in our notation) it has been shown in \cite{Hoffman_2003} that $\Hi_{\mathrm{CK}}(\R)$ and $\Hi_{\mathrm{GL}}(\R)$ are graded duals of each other. The graded dual $V^{\mathrm{gr}}$ of a graded vector space $V$ is defined as the direct sum of the duals of its homogeneous subspaces. This is no longer true if $\R$ is replaced by an infinite dimensional vector space $V$, because in this case the spaces $(\Tree(V)\otimes\Tree(V))^{\mathrm{gr}}$ and $\Tree(V)^{\mathrm{gr}}\otimes\Tree(V)^{\mathrm{gr}}$ are no longer canonically isomorphic. Nevertheless, the Grosmman-Larson and Connes-Kreimer algebras still form a dual pair in a sense which will be made precise in this section.
	
	\begin{defn}
		Let $V$ and $W$ be vector spaces and let $\langle\cdot,\cdot\rangle:V\times W\rightarrow\R$ be a bilinear map. The triple $(V,W,\langle\cdot,\cdot\rangle)$ is called a \emph{dual pair of vector spaces} if
		\begin{enumerate}[(i)]
			\item The set of linear maps $\{\langle\cdot,w\rangle:w\in W\}$ separates $V$ and
			\item the set of linear maps $\{\langle v,\cdot\rangle:v\in V\}$ separates $W$.
		\end{enumerate}
		If $V$ and $W$ are locally convex vector spaces, we say that $(V,W,\langle\cdot,\cdot\rangle)$ is a \emph{dual pair of locally convex vector spaces} if in addition $\langle\cdot,w\rangle$ and $\langle v,\cdot\rangle$ are continuous for all $v\in V$ and $w\in W$. 
	\end{defn}
	
	\begin{rem}
		Any bilinear map $\langle\cdot,\cdot\rangle:V\times W\rightarrow R$ induces a bilinear map on $V^{\otimes k}\times W^{\otimes k}$ for $k\in\Zpp$ which is given by
		\[
			(v_1\otimes\cdots\otimes v_k,w_1\otimes\cdots\otimes w_k)\mapsto \langle v_1,w_1\rangle\cdots\langle v_k,w_k\rangle.
		\]
		We will use this fact below and denote the resulting map also by $\langle\cdot,\cdot\rangle$ without further comment.
	\end{rem}
	
	\begin{lem}\label{lem:TreeProductSeparates}
		Let $\tree\in\Tree$ and let $(V,W,\langle\cdot,\cdot\rangle)$ be a dual pair of vector spaces. Define 
		\begin{equation}\label{eq:TreePairing}
			\langle\cdot,\cdot\rangle_\tree:V^{\otimes\tree}\times W^{\otimes\tree}\rightarrow\R:(\mathbf t_1,\mathbf t_2)\mapsto\sum_{\sigma\in\Sym(\tree)}\langle \sigma(t_1),t_2\rangle=\sum_{\sigma\in\Sym(\tree)}\langle t_1,\sigma(t_2)\rangle
		\end{equation}
		where $t_1\in V^{\otimes\lvert\tree\rvert}$ and $t_2\in W^{\otimes\lvert\tree\rvert}$ are representatives of $\mathbf t_1$ and $\mathbf t_2$ respectively. (Recall that we denote by $\Sym(\tree)$ both the symmetry group of $\tree$ and its embedding into $\mathfrak S_{\lvert\tree\rvert}$ which is induced by the enumeration $\rho_\tree$.) Then the following hold. 
	\begin{enumerate}[(i)]
		\item $\left(V^{\otimes\tree},W^{\otimes\tree},\langle\cdot,\cdot\rangle_\tree\right)$ is a dual pair of vector spaces.
		\item If $(V,W,\langle\cdot,\cdot\rangle)$ is a dual pair of locally convex vector spaces, then $\langle\cdot,\mathbf t_2\rangle_\tree$ can be extended continuously to $V^{\hat\otimes\tree}$ for all $\mathbf t_2\in W^{\otimes\tree}$ and $\left(V^{\hat\otimes\tree},W^{\otimes\tree},\langle\cdot,\cdot\rangle_\tree\right)$ is a dual pair of vector spaces.
	\end{enumerate}

	\end{lem}
	\begin{proof}
	Let us first show that \eqref{eq:TreePairing} does not depend on the choice of $t_1$ or $t_2$. Thus, let $\tilde t_1$  be different representatives of $\mathbf t_1$. Then we have
	\[
		\sum_{\sigma\in\Sym(\tree)}\langle \sigma(t_1),t_2\rangle-\sum_{\sigma\in\Sym(\tree)}\langle \sigma(\tilde t_1),t_2\rangle=\left\langle\sum_{\sigma\in\Sym(\tree)} \sigma(t_1-\tilde t_1),t_2\right\rangle=0
	\]
	and similarly for $t_2$. 
	
	Let us now prove (i). Fix $\mathbf t_2\in W^{\otimes\tree}$
	and assume that we have $\langle \mathbf t_1,\mathbf t_2\rangle_\tree=0$ for all $\mathbf t_1\in V^{\otimes\tree}$. Thus if we fix a representative $t_2$ of $\mathbf t_2$ we have
		\[
			0=\langle \mathbf t_1,\mathbf t_2\rangle_\tree=\sum_{\sigma\in\Sym(\tree)}\langle t_1,\sigma(t_2)\rangle=\left\langle t_1, \sum_{\sigma\in\Sym(\tree)} \sigma(t_2)\right\rangle
		\]
		for all $t_1\in V^{\otimes\lvert\tree\rvert}$ which implies $\sum_{\sigma\in\Sym(\tree)} \sigma(t_2)=0$, i.e. $\mathbf t_2=0$. This last implication holds because it is shown in \cite[Lem. VII/1]{Robertson_Robertson_1980} that $\left(V^{\otimes\lvert\tree\rvert},W^{\otimes\lvert\tree\rvert},\langle\cdot,\cdot\rangle\right)$ is a dual pair. Therefore the maps $\langle\mathbf t_1,\cdot\rangle_\tree$ with $\mathbf t_1\in V^{\otimes\tree}$ separate $W^{\otimes\tree}$. The converse statement is shown in exactly the same way and  hence the proof of (i) is finished.
		
		In order to prove (ii), we equip $V^{\otimes\tree}$ with the weak topology induced by $\langle\cdot,\cdot\rangle_\tree$, i.e. with the weakest topology which makes the maps $\langle\cdot,\mathbf t_2\rangle_\tree$ continuous for all $\mathbf t_2\in W^{\otimes\tree}$. Denote this topology by $\tau_1$. Thus $W^{\otimes\tree}$ is the topological dual of $V^{\otimes\tree}$ with respect to $\tau_1$ and therefore it is also the topological dual of the completion of $V^{\otimes\tree}$ under $\tau_1$. As a consequence $W^{\otimes\tree}$ separates this completion. Since we assume that the pairing between $V$ and $W$ is continuous with respect to their locally convex topologies, $\tau_1$ is weaker than any topology on $V^{\otimes\tree}$ that is induced by cross-norms as introduced in Section~\ref{sec:Topology}. Therefore the space $V^{\hat\otimes\tree}$ is contained in the $\tau_1$-completion of $V^{\otimes\tree}$ and hence $W^{\otimes\tree}$ also separates $V^{\hat\otimes\tree}$. Thus we have shown that $\left(V^{\hat\otimes\tree},W^{\otimes\tree},\langle\cdot,\cdot\rangle_\tree\right)$ is a dual pair.
	\end{proof}
	
	\begin{exmp}
		\[
			\left\langle\tikzTree{15}{\drawRoot\addnodeMiddle{a}{b}{v_1}\addnodeRight{b}{c}{v_3}\addnodeLeft{b}{d}{v_2}},\tikzTree{15}{\drawRoot\addnodeMiddle{a}{b}{w_1}\addnodeRight{b}{c}{w_3}\addnodeLeft{b}{d}{w_2}}\right\rangle
			=
			\langle v_1,w_1\rangle\langle v_2,w_2\rangle\langle v_3,w_3\rangle+
			\langle v_1,w_1\rangle\langle v_3,w_2\rangle\langle v_2,w_3\rangle
		\]
	\end{exmp}
	
	\begin{defn}
		Let $A$ and $B$ be two Hopf algebras. We say that they form a \emph{dual pair of Hopf algebras} with the pairing $\langle\cdot,\cdot\rangle:A\times B\rightarrow\R$ if $(A,B,\langle\cdot,\cdot\rangle)$ is a dual pair of vector spaces which satisfies the conditions
		 \begin{equation}\label{eq:HopfDualPair}
			 \langle\Delta(a),b_1\otimes b_2\rangle=\langle a,b_1b_2\rangle\qquad\text{and}\qquad\langle a_1\otimes a_2,\Delta(b)\rangle=\langle a_1a_2,b\rangle
		 \end{equation}
		for all $a,a_1,a_2\in A$ and $b,b_1,b_2\in B$.
	\end{defn}
	
	\begin{rem}\label{rem:HopfDualAlternative}
		An alternative way to formulate \eqref{eq:HopfDualPair} is the following. Consider the linear maps
		\[
			\varphi:A\rightarrow B^*:a\mapsto\langle a,\cdot\rangle\qquad\text{and}\qquad\psi:B\rightarrow A^*:b\mapsto\langle \cdot,b\rangle,
		\]
		where $A^*$ and $B^*$ denote the (algebraic) duals of $A$ and $B$ respectively. Then \eqref{eq:HopfDualPair} holds if and only if both  $\varphi$ and $\psi$ are algebra morphisms.
	\end{rem}
	
	\begin{lem}\label{lem:groupLike}
		Let $(A,B,\langle\cdot,\cdot\rangle)$ be a dual pair of Hopf algebras. Let $g\in A$ be group-like, i.e. such that $\Delta(g)=g\otimes g$. Then for all elements $b_1,\ldots,b_k\in B$ we have
		\[
			\langle g,b_1\cdots b_k\rangle=\langle g,b_1\rangle\cdots\langle g,b_k\rangle
		\]
	\end{lem}
	\begin{proof}
		Let $k=2$. Then we have
		\[
			\langle g,b_1b_2\rangle=\langle \Delta(g),b_1\otimes b_2\rangle=\langle g\otimes g,b_1\otimes b_2\rangle=\langle g,b_1\rangle\langle g,b_2\rangle.
		\]
		The claim for $k>2$ follows by induction.
	\end{proof}
	
	\begin{thm}\label{thm:HopfDualPair}
		Let $(V,W,\langle\cdot,\cdot\rangle)$ be a dual pair of vector spaces. Define
		\[
			\langle\cdot,\cdot\rangle_\Tree:\Tree(V)\times\Tree(W)\rightarrow\R:(\mathbf t_1,\mathbf t_2)\mapsto \delta_{\tree_1,\tree_2}\langle\mathbf t_1,\mathbf t_2\rangle_{\tree_1}
		\]
		for $\mathbf t_1\in V^{\otimes\tree_1}$ and  $\mathbf t_2\in W^{\otimes\tree_2}$.
		Then the following hold.
		\begin{enumerate}[(i)]
		\item
		$
			\left(\Hi_{\mathrm{GL}}(V),\Hi_{\mathrm{CK}}(W),\langle\cdot,\cdot\rangle_\Tree\right)
		$
		is a dual pair of Hopf algebras. 
		\item If $(V,W,\langle\cdot,\cdot\rangle)$ is a dual pair of locally convex vector spaces, then $\langle\cdot,\mathbf t_2\rangle_\Tree$ can be continuously extended to $\widehat\Tree(V)$ for all $\mathbf t_2\in\Tree(W)$ and both
		$
			\left(\widehat\Hi_{\mathrm{GL}}(V),\Hi_{\mathrm{CK}}(W),\langle\cdot,\cdot\rangle_\Tree\right)
		$
		and 
		$
			\left(\widehat\Hi_{\mathrm{CK}}(V),\Hi_{\mathrm{GL}}(W),\langle\cdot,\cdot\rangle_\Tree\right)
		$
		are dual pairs of Hopf algebras.
		\end{enumerate}
	\end{thm}
	\begin{proof}
		We start with the proof of (i), which means we need to show two things. 
		Firstly we have to make sure that $\Tree(V)$ and $\Tree(W)$ separate each other. 
This follows immediately from Lemma~\ref{lem:TreeProductSeparates}. Secondly we 
have to show that the pairing satisfies \eqref{eq:HopfDualPair}. This has been 
shown in \cite[Proposition 4.4]{Hoffman_2003} for the special case $V=\R$, i.e. 
for unlabelled trees.

		We will extend their arguments to the labelled case. 
		
		Let $\tree_1,\tree_2,\stree\in\Tree$ be trees and let $\mathbf t_1\in V^{\otimes\tree}$, $\mathbf t_2\in V^{\otimes\tree_2}$ and $\mathbf s\in W^{\otimes\stree}$ be elementary trees with elementary representatives $t_1$, $t_2$ and $s$ respectively. 
		
		It is an easy exercise to show that both $\Hi_{\mathrm{GL}}(V)$ and $\Hi_{\mathrm{CK}}(V)$ are generated as an algebra by the elements of $\Tree_1(V)$. In light of Remark~\ref{rem:HopfDualAlternative} it is therefore sufficient to show
		\begin{equation*}
			\langle\mathbf t_1\star \mathbf t_2,\mathbf s\rangle_\Tree=\langle\mathbf t_1\otimes\mathbf t_2,\Delta_{\mathrm{CK}}(\mathbf s)\rangle_\Tree
			\quad\text{and}\quad
			\langle \mathbf t_1 \circ \mathbf t_2, \mathbf s\rangle_\Tree
			= \langle \mathbf t_1 \otimes \mathbf t_2,\Delta_{\mathrm{GL}}(\mathbf s)\rangle_\Tree
		\end{equation*}
		for $\tree_1\in\Tree_1$. By definition the first equality is equivalent to 
		\begin{equation}		
		\begin{aligned}\label{eq:proofDuality}
			&\sum_{d\in\mathrm{Gr}(\tree_1,\tree_2;\stree)}\sum_{\sigma\in\Sym(\stree)}\langle t_1\stackrel{d}{\rightharpoonup}t_2,\sigma(s)\rangle\\
			=&
			\sum_{C\in\mathrm{Cut}(\stree;\tree_1,\tree_2)}\sum_{\sigma_1\in\Sym(\tree_1)}\sum_{\sigma_2\in\Sym(\tree_2)}\langle\sigma_1(t_1)\otimes\sigma_2(t_2),\mathbf P^C(s)\otimes\mathbf R^C(s)\rangle,
		\end{aligned}
		\end{equation}
		where we denote by $\mathrm{Gr}(\tree_1,\tree_2;\stree)$ the set of all grafting maps $d$ such that $\tree_1\stackrel{d}{\rightharpoonup}\tree_2=\stree$ and by $\mathrm{Cut}(\stree;\tree_1,\tree_2)$ the set of all cuts $C$ such that $P^C(\stree)=\tree_1$ and $R^C(\stree)=\tree_2$.
		
		Denote by $c$ the unique child of the root of $\tree_1\in\Tree_1$. Any $d\in\mathrm{Gr}(\tree_1,\tree_2)$ is uniquely determined by the vertex of $\tree_2$ to which $c$ is attached. Let us now rewrite the left-hand side of \eqref{eq:proofDuality}. Instead of summing over all $d\in\mathrm{Gr}(\tree_1,\tree_2;\stree)$ we can alternatively fix an arbitrary $\bar d\in\mathrm{Gr}(\tree_1,\tree_2;\stree)$, denote the associated vertex of $\tree_2$ by $v$ and rewrite the expression as
		\begin{equation}\label{eq:dualityGraftRewritten}
		\frac{1}{\lvert\mathrm{Fix}(v,\tree_2)\rvert}\sum_{\pi\in \Sym(\tree_2)}\sum_{\sigma\in\Sym(\stree)}\langle t_1\stackrel{\bar d}{\rightharpoonup}\pi(t_2),\sigma(s)\rangle,
		\end{equation}
		where $\mathrm{Fix}(v,\tree_2):=\{\sigma\in\Sym(\tree_2):\sigma(v)=v\}$. In order to verify that \eqref{eq:dualityGraftRewritten} is indeed equal to the left hand side of \eqref{eq:proofDuality} we observe that
		\[
			\sum_{\sigma\in\Sym(\stree)}\langle t_1\stackrel{\bar d}{\rightharpoonup}\pi(t_2),\sigma(s)\rangle=
			\sum_{\sigma\in\Sym(\stree)}\langle t_1\stackrel{\bar d}{\rightharpoonup}\tilde\pi(t_2),\sigma(s)\rangle
		\]
		holds for $\pi,\tilde\pi\in\Sym(\tree_2)$ whenever $\pi^{-1}(v)=\tilde\pi^{-1}(v)$. For fixed $\pi\in\Sym(\tree_2)$ there are $\lvert\mathrm{Fix}(v,\tree_2)\rvert$ ways to choose $\tilde\pi$ in such a way. 
		
		On the other hand, any cut $C\in\mathrm{Cut}(\stree;\tree_1,\tree_2)$ contains exactly one vertex and hence we can choose $\overline C\in\mathrm{Cut}(\stree;\tree_1,\tree_2)$ such that for the canonical injection $\phi_{\overline C,P}:\mathrm{v}(\tree_1)\rightarrow\mathrm{v}(\stree)$ we have that $w:=\phi_{\overline C,R}(c)$ is the unique vertex which defines $\overline C$. Thus we can rewrite the right-hand side of \eqref{eq:proofDuality} as
		\begin{equation}\label{eq:dualityCutRewritten}		
			\frac{1}{\lvert\mathrm{Fix}(w,\stree)\rvert}
			\sum_{\tau\in\Sym(\stree)}
			\sum_{\sigma_1\in\Sym(\tree_1)}
			\sum_{\sigma_2\in\Sym(\tree_2)}
			\left\langle\sigma_1(t_1)\otimes\sigma_2(t_2),\mathbf{P}^{\overline C}(\tau(s))\otimes\mathbf{R}^{\overline C}(\tau(s))\right\rangle,
		\end{equation}
		where $\mathrm{Fix}(w,\stree):=\{\sigma\in\Sym(\stree):\sigma(w)=w\}$. This can verified as in the previous paragraph. Since $\phi_{\overline C,P}$ induces a canonical embedding of $\Sym(\tree_1)$ into $\Sym(\stree)$ as a subgroup we can rewrite \eqref{eq:dualityCutRewritten} as
		\begin{align*}
			&\frac{\lvert\Sym(\tree_1)\rvert}{\lvert\mathrm{Fix}(w,\stree)\rvert}
			\sum_{\tau\in\Sym(\stree)}
			\sum_{\sigma_2\in\Sym(\tree_2)}
			\left\langle t_1\otimes\sigma_2(t_2),\mathbf{P}^{\overline C}(\tau(s))\otimes\mathbf{R}^{\overline C}(\tau(s))\right\rangle\\
		=&\frac{\vert\Sym(\tree_1)\rvert}{\lvert\mathrm{Fix}(w,\stree)\rvert}
		\sum_{\tau\in\Sym(\stree)}
		\sum_{\sigma_2\in\Sym(\tree_2)}
		\left\langle t_1\stackrel{\overline d}\rightharpoonup\sigma_2(t_2),\tau(s)\right\rangle
		\end{align*}
		One easily verifies that $\lvert\mathrm{Fix}(w,\stree)\rvert=\lvert\mathrm{Fix}(v,\tree_2)\rvert\lvert\Sym(\tree_1)\rvert$ holds, which shows \eqref{eq:proofDuality}. 
		
		The second condition in \eqref{eq:HopfDualPair}, i.e. the relationship between $\circ$ and $\Delta_{\mathrm{GL}}$, is shown analogously.

		Finally, (ii) is an easy consequence of (i). 
	\end{proof}
	
	\begin{cor}
		Let $V$ be a locally convex vector space and let $\mathbf t\in\widehat\Hi_\mathrm{GL}(V)$ be group-like i.e. such that it satisfies $\Delta_{\mathrm{GL}}(\mathbf t)=\mathbf t\otimes\mathbf t$. Then we have
		\[
			 \mathbf t = \exp_\circ(\pi_1(\mathbf t)):=\sum_{k=0}^\infty\frac{(\pi_1(\mathbf t))^{\circ k}}{k!}.
		\]
	\end{cor}
	\begin{proof}
		Let $V'$ be the continuous dual of $V$, let $\stree\in\Tree$ and let $\mathbf s\in(V')^{\otimes\stree}$ be an elementary labelled tree with decomposition $\mathbf s=\mathbf s_1\circ\cdots\circ\mathbf s_k$ as in \eqref{eq:decompositionLabelled}. Then we have
		\begin{align*}
			\langle\mathbf t,\mathbf s\rangle_\Tree
			&=\langle\mathbf t, \mathbf s_1\circ\cdots\circ \mathbf s_k\rangle_\Tree\\
			&=\langle\mathbf t,\mathbf s_1\rangle_\Tree\cdots\langle\mathbf t,\mathbf s_k\rangle_\Tree\\
			&=\langle \pi_1(\mathbf t),\mathbf s_1\rangle_\Tree\cdots\langle \pi_1(\mathbf t),\mathbf s_k\rangle_\Tree,
		\end{align*}
	where we have used first Lemma~\ref{lem:groupLike} in combination with Theorem~\ref{thm:HopfDualPair} and then the fact that $\mathbf s_i\in\Tree_1(V)$ for all $i\in\{1,\ldots,k\}$. On the other hand we have
	\begin{align*}
		\left\langle(\pi_1(\mathbf t))^{\circ k}, \mathbf s\right\rangle_\Tree		
		=&\left\langle (\pi_1(\mathbf t))^{\circ k}, \mathbf s_1\circ\cdots\circ \mathbf s_k\right\rangle_\Tree\\
		=&k!\langle \pi_1(\mathbf t),\mathbf s_1\rangle_\Tree\cdots\langle \pi_1(\mathbf t),\mathbf s_k\rangle_\Tree,
	\end{align*}	
	which follows from \eqref{eq:TreePairing} and the multinomial formula.	
	
	Since we have chosen an arbitrary $\mathbf s$ and since $\Tree(V')$ separates the points of $\widehat\Tree(V)$ this shows that we have $\pi_k(\mathbf t)=\frac1{k!}\left(\pi_1(\mathbf t)\right)^{\circ k}$ for all $k\in\Zp$. Hence the claim follows.
	\end{proof}
	
	\begin{cor}\label{cor:truncatedDuality}
		Let $V$ be a locally convex vector space and let $\mathbf t\in\widehat\Hi_\mathrm{GL}(V)$ be group-like i.e. such that it satisfies $\Delta_{\mathrm{GL}}(\mathbf t)=\mathbf t\otimes\mathbf t$. Then we have
		\[
			 \pi_k^{(n)}(\mathbf t) = \frac{\left(\pi_1^{(n)}(\mathbf t)\right)^{\circ k}}{k!}+\mathfrak r,
		\]
		where $\mathfrak r$ is such that $\pi^{(n)}(\mathfrak r)=0$.
	\end{cor}
	\begin{proof}
		This follows immediately from the previous corollary.
	\end{proof}
	
	\section{Change of variables for differential equations driven by branched rough paths in arbitrary dimension}\label{sec:RoughPaths}
	
	For a general introduction to the theory of rough paths we refer to \cite{Lyons_Caruana_Levy_2007,Friz_Hairer_2014,Friz_Victoir_2010}. In this section we assume that $E$ is a Banach space, which implies that $\widehat\Tree^{(k)}(E)$ is also Banach for all $k\in\Zpp$. For $T>0$ define $\Delta_T:=\{(s,t)\in[0,T]:s<t\}$. By a \emph{control function} or simply \emph{control} we mean a continuous function $\omega:\Delta_T\rightarrow\R$ which satisfies $\omega(s,s)=0$ and $\omega(s,t)+\omega(t,u)\leq\omega(s,u)$ for all $0\leq s\leq t\leq u\leq T$.

	\begin{defn}
		Let $T>0$, let $p\geq 1$ and let $\omega$ be a control function. A continuous map $\bx:\Delta_T\rightarrow\widehat\Tree^{([p])}(E)$ is called a \emph{branched $p$-rough path controlled by $\omega$} if it satisfies
		\begin{enumerate}[(i)]
			\item $\bx_{s,u}=\bx_{s,t}\star\bx_{t,u}$
			\item $\bx_{s,t}=\pi^{([p])}(\mathbf g)$ for some $\mathbf g\in \widehat\Tree(E)$ with $\Delta_{\mathrm{GL}}(\mathbf g)=\mathbf g\otimes\mathbf g$
			\item $\left\lVert\pi^k(\bx_{s,t})\right\rVert\leq\omega(s,t)^{\frac kp}$
		\end{enumerate}
		for all $0\leq s\leq t\leq u\leq T$ and $k\in\{1,\ldots,[p]\}$. 
	\end{defn}	
	
	Let $F$ be another Banach space and let $V:E\times F\rightarrow F$ be a continuous map which is linear in the first argument and smooth (in the Fr\'echet sense) in the second argument. There are two ways to view such a map. The one that is chosen in \cite{Lyons_Caruana_Levy_2007} is to interpret $V$ as a map from $F$ to the space of linear maps from $E$ to $F$. However, we prefer to view $V$ as a continuous linear map from $E$ into the space of smooth vector fields on $F$. To emphasise this point of view we write $V_e$ instead of $V(e,\cdot)$ for $e\in E$. Note that the the vector fields $V_e$ for $e\in E$ are \emph{kinetic vector fields} in the language of \cite{Kriegl_Michor_1997}.
	
	Let furthermore $f:F\rightarrow G$ be a smooth map into some Banach space $G$. Then for $y\in F$ the linear map
	\[
		\Psi_{V,f,y}:\widehat\Tree(E)\rightarrow G
	\]
	is constructed in the following way. It maps a labelled tree $\mathbf t$ to a product of derivatives of $f$ and $V$, evaluated at $y$. The root of $\mathbf t$ corresponds to $f$ and a vertex labelled by $e\in E$ corresponds to the vector field $V_e$. The number of children of a vertex indicates how often the corresponding term has to be differentiated. The terms corresponding to the children are then plugged into the derivative.
	\begin{exmp}
		\[
		\Psi_{V,f,y}\left(\kern-.5em\tikzTree{15}{\drawRoot\addnodeLeftTwo{a}{b}{e_1}\addnodeRightTwo{a}{c}{e_2}\addnodeLeft{b}{d}{e_3}\addnodeMiddle{b}{e}{e_4}\addnodeRight{b}{f}{e_5}}\right)
		=
		f''(y)[V_{e_1}'''(y)[V_{e_3}(y),V_{e_4}(y),V_{e_5}(y)],V_{e_2}(y)]		
		\]
	\end{exmp}
	
	There is a particular choice of $f$ which will become important. For any positive integer $k$ we define the map 
	\[
		\id^k:F\rightarrow F^{\otimes k}:x\mapsto x^{\otimes k}.
	\]
 The maps $\Psi_{V,\id^k,y}$ have two important properties. First they satisfy 
 	\begin{equation}\label{eq:property1}
		\Psi_{V,f,y}(\mathbf t)=\frac1{k!}f^{(k)}(y)[\Psi_{V,\id^k,y}(\mathbf t)]
	\end{equation} for all $\mathbf t\in\widehat\Tree(E)$. Second they are compatible with the $\circ$-product of labelled trees in the sense that for two trees $\mathbf t_1\in\widehat\Tree_k(E)$ and $\mathbf t_2\in\widehat\Tree_l(E)$ we have
	\begin{equation}\label{eq:property2}
		\Psi_{V,\id^{k+l},y}(\mathbf t_1\circ\mathbf t_2)=\frac{(k+l)!}{k!l!}\sym_{\mathfrak S_{k+l}}\left(\Psi_{V,\id^k,y}(\mathbf t_1)\otimes\Psi_{V,\id^l,y}(\mathbf t_2)\right).
	\end{equation}
	
	\begin{thm}\label{thm:changeOfVariable}
		Let $V:E\times F\rightarrow F$ be as above. Let furthermore $\bx$ be a branched $p$-rough path controlled by $\omega$ and let $y:[0,T]\rightarrow F$ be a continuous path. Then the following are equivalent.
		\begin{enumerate}[(i)]
			\item There exists a function $\theta:\R\rightarrow\R$ with $\theta(h)/h\rightarrow0$ for $h\rightarrow0$ such that
			\begin{equation}
				\label{eq:changeOfVariableId}
				\lVert y_t-\Psi_{V,\id,y_s}(\bx_{s,t})\rVert\leq\theta(\omega(s,t))
			\end{equation}
		holds for all $(s,t)\in\Delta_T$.
			\item For every Banach space $G$ and every smooth $f:F\rightarrow G$ there exists a function $\theta_f:\R\rightarrow\R$ with $\theta_f(h)/h\rightarrow0$ for $h\rightarrow0$ such that
			\begin{equation}
				\label{eq:changeOfVariableF}
				\lVert f(y_t)-\Psi_{V,f,y_s}(\bx_{s,t})\rVert\leq\theta_f(\omega(s,t))
			\end{equation}
			holds for all $(s,t)\in\Delta_T$. 
		\end{enumerate}
	\end{thm}
	
	\begin{proof}
		The implication $(ii)\Rightarrow(i)$ is trivial. For the other direction, the idea is to split the term $\Psi_{V,f,y_s}(\bx_{s,t})$ into terms that we can control. We have (see explanations below)
		\refstepcounter{equation}
		\label{eq:proofChainRule}		
		\begin{align}
			&\left\lVert f(y_t)-\Psi_{V,f,y_s}(\bx_{s,t})\right\rVert\nonumber\\
			\leq&\left\lVert\sum_{k=1}^{\lfloor p\rfloor}\frac{f^{(k)}(y_s)}{k!}\left[y_{s,t}^{\otimes k}-\Psi_{V,\id^k,y_s}(\pi_k(\bx_{s,t}))\right]\right\rVert+C_1\lVert y_{s,t}\rVert^{\lfloor p\rfloor+1}\tag{\theequation a}\label{eq:lineA}\\
			\leq&\left\lVert\sum_{k=1}^{\lfloor p\rfloor}\frac{f^{(k)}(y_s)}{k!}\left[y_{s,t}^{\otimes k}-\frac1{k!}\Psi_{V,\id^k,y_s}\left((\pi_1(\bx_{s,t}))^{\circ k}\right)\right]\right\rVert+C_1\lVert y_{s,t}\rVert^{\lfloor p\rfloor+1}+C_2\omega(s,t)^{\frac{\lfloor p\rfloor+1}{p}}\tag{\theequation b}\label{eq:lineB}\\
			=&\left\lVert\sum_{k=1}^{\lfloor p\rfloor}\frac{f^{(k)}(y_s)}{k!}\left[y_{s,t}^{\otimes k}-\left(\Psi_{V,\id,y_s}\left(\pi_1(\bx_{s,t})\right)\right)^{\otimes k}\right]\right\rVert+C_1\lVert y_{s,t}\rVert^{\lfloor p\rfloor+1}+C_2\omega(s,t)^{\frac{\lfloor p\rfloor+1}{p}}\tag{\theequation c}\label{eq:lineC}\\
			=&\left\lVert\sum_{k=1}^{\lfloor p\rfloor}\frac{f^{(k)}(y_s)}{k!}\left[\left(y_{s,t}-\Psi_{V,\id,y_s}\left(\pi_1(\bx_{s,t})\right)\right)\otimes\sum_{i=0}^{k-1}y_{s,t}^{\otimes (k-1-i)}\Psi_{V,\id,y_s}\left(\pi_1(\bx_{s,t})\right)^{\otimes i}\right]\right\rVert\nonumber\\
			&\qquad+C_1\lVert y_{s,t}\rVert^{\lfloor p\rfloor+1}+C_2\omega(s,t)^{\frac{\lfloor p\rfloor+1}{p}}
			\tag{\theequation d}\label{eq:lineD}\\
			=&\left\lVert\sum_{k=1}^{\lfloor p\rfloor}\frac{f^{(k)}(y_s)}{k!}\left[\left(y_t-\Psi_{V,\id,y_s}\left(\bx_{s,t}\right)\right)\otimes\sum_{i=0}^{k-1}y_{s,t}^{\otimes (k-1-i)}\Psi_{V,\id,y_s}\left(\pi_1(\bx_{s,t})\right)^{\otimes i}\right]\right\rVert\nonumber\\
			&\qquad+C_1\lVert y_{s,t}\rVert^{\lfloor p\rfloor+1}+C_2\omega(s,t)^{\frac{\lfloor p\rfloor+1}{p}}\tag{\theequation e}\label{eq:lineE}.
		\end{align}
		In this computation we have used the following arguments.
		\begin{enumerate}[(i)]
		\item[\eqref{eq:lineA}] Replace $f$ by its Taylor approximation up to degree $\lfloor p\rfloor$ and make use of \eqref{eq:property1}.
		\item[\eqref{eq:lineB}] Apply Corollary~\ref{cor:truncatedDuality}, where the remainder $\mathfrak r$ leads to the term that involves $\omega$.
		\item[\eqref{eq:lineC}] Use \eqref{eq:property2}, where we can ignore the symmetrisation because $f^{(k)}(y_s)$ is a symmetric multilinear form anyway.
		\item[\eqref{eq:lineD}] Use the elemetary identity $a^k-b^k=(a-b)(a^{k-1}+a^{k-2}b+\cdots+ab^{k-2}+b^{k-1})$, which holds in principle only for commuting variables $a$ and $b$. It can be applied nevertheless, again because $f^{(k)}(y_s)$ is symmetric.
		\item[\eqref{eq:lineE}] Observe that $\Psi_{V,\id,y_s}$ maps any tree which is neither in $\widehat\Tree_1(E)$ nor in $E^{\otimes\mathbf 1}$ to zero.
		\end{enumerate}
				
		Now define the constant $C_3$ by
		\[
			C_3:=\sup_{0\leq s<t\leq T}\sup_{a\in F:\lVert a\rVert=1}\left\lVert\sum_{k=1}^{\lfloor p\rfloor}\frac{f^{(k)}(y_s)}{k!}\left[a\otimes\sum_{i=0}^{k-1}y_{s,t}^{\otimes (k-1-i)}\Psi_{V,\id,y_s}\left(\pi_1(\bx_{s,t})\right)^{\otimes i}\right]\right\rVert.
		\]
		We have $C_3<\infty$ because $f^{(k)}(y)$ is a continuous multilinear map for all $k\in\Zpp$ and all $y\in F$. Thus \eqref{eq:lineE} can be estimated by 
		\begin{align*}
			&C_3\lVert y_t-\Psi_{V,\id,y_s}\left(\bx_{s,t}\right)\rVert+C_1\lVert y_{s,t}\rVert^{\lfloor p\rfloor+1}+C_2\omega(s,t)^{\frac{\lfloor p\rfloor+1}{p}}\\
			\leq&C_3\theta(\omega(s,t))+C_4\omega(s,t)^\frac{\lfloor p\rfloor+1}{p},
		\end{align*}
		for some $C_4>0$ where we have also used that $y$ has finite $p$-variation controlled by (a scalar multiple of) $\omega$. Defining 
		\[
			\theta_f(h):=C_3\theta(h)+C_4h^\frac{\lfloor p\rfloor+1}{p}
		\]
		finishes the proof.
	\end{proof}
	
	\begin{defn}
		\label{def:RDE}
		With the notation as in Theorem~\ref{thm:changeOfVariable} we say that $y$ solves the rough differential equation
		\begin{equation}
			\label{eq:RDE}
			\dd y_t=V(y_t)\dd\bx_t
		\end{equation}
		with initial value $y_0$ if (i) (and thus also (ii)) of Theorem~\ref{thm:changeOfVariable} holds.
	\end{defn}
	
	The idea to define solutions of rough differential equations via Euler expansions is originally due to Davie~\cite{Davie_2008} and was later developed  further in \cite{Friz_Victoir_2010}. Note that, in contrast to Lyons' original work \cite{Lyons_1998,Lyons_Caruana_Levy_2007} and Gubinelli's theory of controlled rough paths \cite{Gubinelli_2004}, this approach defines a solution as an ordinary $F$-valued path without any higher order terms. 
	
	\subsection{Geometric rough paths and differential geometry}
	Even though we have formulated Theorem~\ref{thm:changeOfVariable} only for branched rough paths it is also valid for (weakly) geometric rough path since the latter can be interpreted as branched rough paths in a canonical way. The relationship between branched and geometric rough paths has been discussed in detail in \cite{Hairer_Kelly_2015}, at least for the finite-dimensional case. The crucial difference between branched and geometric rough paths is that geometric rough paths take their values in the tensor algebra over $E$ rather than the Grossman-Larson algebra. 
	
	Let us define the tensor algebra over $E$ and see how it embeds into the Grossman-Larson algebra. For the algebraic details we refer to \cite{Lyons_Caruana_Levy_2007,Reutenauer_1993}. Recall that we have fixed a symmetric system of cross-norms on the tensor powers of $E$ in order to define $\widehat\Tree(E)$. Thus we can equip
	\[
		T(E):=\bigoplus_{k=0}^\infty V^{\hat\otimes k}
	\]
	with the product topology and denote by $\widehat T(E)$ its completion. The space $\widehat T(E)$ carries a natural product, namely the tensor product, which turns it into an algebra. Note that the sub-algebra generated by $E$ is dense in $\widehat T(E)$. Thus we can define the coproduct $\Delta_\shuffle$ on $E$ by
	\[
		\Delta_\shuffle(e):=1\otimes e+e\otimes 1
	\]
	and then extend it to all of $\widehat T(E)$ as a continuous algebra morphism.
	
	Now define \renewcommand{\treeNodeRadius}{1.5pt}$\tree=\kern-4pt\tikzTreeSmall{0}{\drawRoot\addnodeMiddle{a}{b}{}}$\renewcommand{\treeNodeRadius}{2pt} and note that $E^{\otimes\tree}\simeq E$. Hence we can define the continuous algebra morphism
	\[
		\Phi:\widehat T(E)\rightarrow\widehat\Hi_{\mathrm{GL}}(E)
	\]	
	which is given by $\Phi(e)=e\in E^{\otimes\tree}$ for $e\in E$. One easily checks that this is in fact a Hopf algebra morphism. Thus $\Phi$ lets us embed the space of geometric rough paths into the space of branched rough paths.
	
	In \cite{Hairer_Kelly_2015} it has in fact been shown that conversely, one can interpret every branched rough path as a geometric one, albeit this geometric rough path will then take values in the tensor algebra over an \emph{augmented} version of the original vector space. 
	
	Now assume that $V$ is a linear map from $E$ into the space of  vector fields on some smooth manifold $\M$. Let furthermore $y\in\M$ and let $f:\M\rightarrow G$ a smooth map into some Banach space $G$. In this case the `non-geometricity' of general branched rough paths becomes apparent in the definition of the map $\Psi_{V,f,y}$ that we have defined above. Let $e_1,\ldots,e_k\in E$. Then we have
	\[
		\Psi_{V,f,y}\left(\Phi(e_1\otimes\cdots\otimes e_k)\right)
		=
		\Psi_{V,f,y}\left(\renewcommand{\treeNodeRadius}{1.5pt}\kern-4pt\tikzTreeSmall{0}{\drawRoot\addnodeMiddle{a}{b}{e_1}}\ \star\ \cdots\ \star\kern-3pt\tikzTreeSmall{0}{\drawRoot\addnodeMiddle{a}{b}{e_k}}\renewcommand{\treeNodeRadius}{2pt}\right)
		=		
		V_{e_1}\cdots V_{e_k}f(y),
	\]
	because the $\star$-product simply encodes the product rule. Thus when we say that a rough path $\bx$ is geometric, this is because $\Psi_{V,f,y}(\Phi(\bx_{s,t}))$ is a canonically defined object for all $V,f,y,s,t$, even in a non-linear setting. In contrast, if $\bx$ is a branched rough path, then $\Psi_{V,f,y}(\bx_{s,t})$ is in general only well-defined once we choose an affine connection on $\M$ because it involves covariant derivatives of vector fields and higher order derivatives of $f$. Hence it is the independence of the choice of a connection which makes the geometric rough paths special among all branched rough paths. Differential equations driven by geometric rough paths can be interpreted as generalised Stratonovich equations whereas differential equations driven by branched rough paths should be seen as a generalisation of Ito-corrected Stratonovich equations.
	
	In light of the previous paragraph we will write $\Psi_{V,f,y,\nabla}$ in order to make clear that the values of the function on arbitrary labelled trees are to be computed with respect to the connection $\nabla$. To align this with the notation used above, we agree that $\nabla$ is the canonical connection of a vector space if it is not mentioned explicitly. Thus we generalise Definition~\ref{def:RDE} in the following way.
	\begin{defn}
		\label{def:RDE_connection}
		With the notation as in Theorem~\ref{thm:changeOfVariable} we say that $y$ solves the rough differential equation 
		\begin{equation}
			\label{eq:RDE_connection}
			\dd y_t=V(y_t)\dd\bx_t
		\end{equation}
		\emph{with respect to the connection $\nabla$} on $F$ and with initial value $y_0$ if there exists a function $\theta:\R\rightarrow\R$ with $\theta(h)/h\rightarrow0$ for $h\rightarrow0$ such that
			\begin{equation}
				\label{eq:changeOfVariableId_connection}
				\lVert y_t-\Psi_{V,\id,y_s,\nabla}(\bx_{s,t})\rVert\leq\theta(\omega(s,t))
			\end{equation}
		holds for all $(s,t)\in\Delta_T$.
	\end{defn}
	
	To conclude, let us mention two consequences of Theorem~\ref{thm:changeOfVariable} which underline its significance in the context of differential geometry. First, we have the invariance of rough differential equations under changes of coordinates as stated in the following corollary.
	
	\begin{cor}
		Let the notation be as in Theorem~\ref{thm:changeOfVariable} and assume  that $y_t\in U$ for some open set $U\subset F$ and all $t\in[0,T]$. Let $W\subset F$ be another open set and let $\varphi:U\rightarrow W$ be a diffeomorphism. For each $e\in E$ define the vector field $\tilde V_e$ on $W$ by  
		\[
			\tilde V_e(p)=\varphi'(\varphi^{-1}(p))[V_e(\varphi^{-1}(p))].
		\]
		 Then $y$ solves \eqref{eq:RDE_connection} (with respect to the canonical connection on $F$)  if and only if $\tilde y$ given by $\tilde y_t:=\varphi(y_t)$ solves 
		\[
			\dd\tilde y_t=\tilde V(\tilde y_t)\dd\bx_t,
		\]
		with respect to the pushforward connection $\tilde\nabla$ given by 
		\[
			\tilde\nabla_{\tilde Y}\tilde{X}(\tilde p):=\tilde X'(\tilde p)[\tilde Y(\tilde p)]-\varphi''(\varphi^{-1}(\tilde p))[(\varphi'(\varphi^{-1}(\tilde p)))^{-1}[\tilde X(\tilde p)],(\varphi'(\varphi^{-1}(\tilde p)))^{-1}[\tilde Y(\tilde p)]]
		\]
		for vector fields $\tilde X,\tilde Y$ on $F$ and $\tilde p\in W$.
	\end{cor}	
	\begin{proof}
		Writing $\varphi$ itself in the new coordinates on $W$ gives $\tilde\varphi(\tilde p)=\varphi(\varphi^{-1}(\tilde p))=\tilde p$, i.e. $\tilde\varphi=\id$. Thus we have
		\[
			\tilde y_t-\Psi_{\tilde V, \id, \tilde y_s, \tilde \nabla}
				(\bx_{s,t})
				=
				\varphi(y_t)-\Psi_{\tilde V, \tilde\varphi, \tilde y_s, \tilde\nabla}(\bx_{s,t})
				=\varphi(y_t)-\Psi_{V, \varphi, y_s}(\bx_{s,t}),
		\]
		where the last equality is obtained by writing everything in the original coordinates. The claim follows now immediately from Theorem~\ref{thm:changeOfVariable}.
	\end{proof}
	A careful reading of \cite{Cass_Driver_Litterer_2015} shows that the authors have used arguments similar to those that went into the proof of Theorem~\ref{thm:changeOfVariable} in order to obtain well-definedness of RDEs on submanifolds of $\Rd$, albeit in the less general case $p<3$.
		
	Second, Theorem~\ref{thm:changeOfVariable} yields a way to \emph{define} what we mean by the solution of an RDE on a manifold by testing \eqref{eq:changeOfVariableF} for a suitable set of real-valued test functions. If the manifold is finite-dimensional, then smooth compactly supported test functions are an obvious choice. Even though this definition has been suggested before -- see e.g. \cite{Angst_Bailleul_Tardif_2015} and the references therein -- we believe that Theorem~\ref{thm:changeOfVariable} is so far the most general way to establish that such a definition is compatible with Definitions~\ref{def:RDE} and \ref{def:RDE_connection}. 

	\bibliography{../../bibtex/bib_final}
	\bibliographystyle{amsplain}
\end{document}